\documentclass[11pt]{amsart}
\usepackage{amssymb,amsmath,amsfonts,amsthm}
\usepackage{graphicx}
\usepackage{color}
\usepackage{psfrag}
\usepackage{hyperref}

\newtheorem{theorem}{Theorem}[section]
\newtheorem{lemma}[theorem]{Lemma}
\newtheorem{proposition}[theorem]{Proposition}
\newtheorem{corollary}[theorem]{Corollary}

\theoremstyle{definition}
\newtheorem{definition}[theorem]{Definition}

\newtheorem{remark}[theorem]{Remark}

\newtheorem*{theorema}{Theorem A}
\newtheorem*{theoremb}{Theorem B}
\newtheorem*{conjecture}{Conjecture}

\newtheorem{ltheorem}{Theorem} 

\def\real{\mathbb{R}}

\def\integer{\mathbb{Z}}
\def\natural{\mathbb{N}}

\def\d{\operatorname{d}}
\def\vol{\operatorname{vol}}

\def\id{\operatorname{id}}

\def\cV{\mathcal{V}}

\def\cH{\mathcal{H}}

\def\cE{\mathcal{E}}

\def\nB{\mathbf{B}}

\def\hm{\hat{m}}
\def\hmu{\hat{\mu}}

\def\hTheta{{\hat{\Theta}}}

\def\tmu{\tilde{\mu}}

\def\leb{{\operatorname{leb}}}
\def\Dif{\operatorname{Diff}}
\def\SP{SP^{r,\alpha}_{\sigma,\mu}}
\def\SPV{SP^{r}_{\leb}(M)}

\def\tf{{\tilde{f}}}

\def\tM{{\tilde{M}}}
\def\tv{{\tilde{v}}}
\def\tq{{\tilde{q}}}
\def\tz{{\tilde{z}}}
\def\tx{{\tilde{x}}}
\def\ty{{\tilde{y}}}
\def\tm{{\tilde{m}}}
\def\tp{{\tilde{p}}}
\def\hz{{\tilde{z}}}
\def\hp{{\tilde{p}}}
\def\hx{{\tilde{x}}}
\def\hy{{\tilde{y}}}

\newcommand{\norm}[1]{{\left\lVert  #1  \right\rVert}}
\newcommand{\abs}[1]{{\left\lvert  #1  \right\rvert}}

\title[Generic positive exponents]{On the genericity of positive exponents of conservative skew products with two-dimensional fibers}
\author{Davi Obata}
\author{Mauricio Poletti}

\newcommand{\information}{{
  \bigskip
  \footnotesize
	\textbf{Davi Obata}: \textsc{CNRS-Laboratoire de Math\'ematiques d'Orsay, UMR 8628, Universit\'e Paris-Sud 11, Orsay Cedex 91405, France } \par\nopagebreak
  \textsc{Instituto de Matem\'atica, Universidade Federal do Rio de Janeiro, P.O. Box 68530, 21945-970, Rio de Janeiro Brazil}\par\nopagebreak
  \textit{E-mail:} \texttt{davi.obata@math.u-psud.fr}

  \medskip
  \textbf{Mauricio Poletti}: \textsc{CNRS-Laboratoire de Math\'ematiques d'Orsay, UMR 8628, Universit\'e Paris-Sud 11, Orsay Cedex 91405, France } \par\nopagebreak
  \textit{E-mail:} \texttt{mpoletti@impa.br}
}}

\begin{document}

\begin{abstract}
In this paper we study the existence of positive Lyapunov exponents for three different types of skew products, whose fibers are compact Riemannian surfaces and the action on the fibers are by volume preserving diffeomorphisms. These three types include skew products with a volume preserving Anosov diffeomorphism on the basis; or with a subshift of finite type on the basis preserving a measure with product structure; or locally constant skew products with Bernoulli shifts on the basis. We prove the $C^1$-density and $C^r$-openess of the existence of positive Lyapunov exponents on a set of positive measure in the space of such skew products.    
\end{abstract}

\maketitle

\section{Introduction}

In the $60$'s, Smale had obtained several results about dynamical implications of uniform hyperbolicity (see \cite{Sm67}). Since then, uniform hyperbolic dynamics have been very well understood. For instance, hyperbolic transitive sets have several features, such as a symbolic dynamics associated to it, existence of periodic points and horseshoes, positive entropy, ergodicity (in the volume preserving scenario), etc. Even though, uniform hyperbolicity is a $C^1$-open property, it is not a $C^1$-dense property.

For invariant measures, Pesin proposed in \cite{Pes77} a weaker notion of hyperbolicity, called non-uniform hyperbolicity. A diffeomorphism $f$ that preserves a probability measure $\mu$ is called non-uniformly hyperbolic if all its Lyapunov exponents are non-zero (see \cite{BaP01} for precise definitions). It turns out that non-uniform hyperbolicity also imply several interesting features of the dynamics, such as existence of periodic orbits and horseshoes \cite{Ka80}, countably many ergodic components for smooth measures \cite{Pes77}, etc.

Given a probability measure $\mu$ on a compact Riemannian manifold $M$, one can consider the space of $C^r$-diffeomorphisms that preserves this measure $\mathrm{Diff}^r_{\mu}(M)$. A natural question is to know how frequent is non-uniform hyperbolicity in $\mathrm{Diff}^r_{\mu}(M)$?

There are several results related to this question, most of them for the case when $\mu$ is the Lebesgue measure. For instance, a remarkable result by Ma\~n\'e \cite{Man96} and Bochi \cite{Boc02} proved that on surfaces any area preserving diffeomorphism which is not Anosov can be $C^1$-approximated by an area preserving diffeomorphism with some zero Lyapunov exponent (see section \ref{sec.preliminaries} for the definition of Anosov diffeomorphism). 

In this paper we address this question for some skew products over hyperbolic maps. Let us define the scenarios we will be working with.

Let $\tM$ be a smooth, compact, connected and oriented manifold and $S$ be a smooth, compact and connected surface. Consider a fiber bundle $M$ over $\tM$, defined by a smooth projection $\pi: M \to \tM$, with fibers diffeomorphic to $S$. For a point $x\in M$, we write $S_x$ the fiber that contains the point $x$. We say that a diffeomorphism $f:M \to M$ \emph{preserves fibers} if for any $x\in M$ it holds $S_{f(x)} = f(S_x)$. 

A diffeomorphism $f$ is \emph{partially hyperbolic} if there is a $Df$-invariant decomposition of the tangent bundle $TM = E^s \oplus E^c \oplus E^u$, such that $Df|_{E^s}$ contracts, $Df|_{E^u}$ expands and the behavior of $Df|_{E^c}$ is bounded by the contraction along $E^s$ and the expansion along $E^u$, see section \ref{subsection.ph} for a precise definition.

For the fiber bundle $M$ a diffeomorphism $f:M \to M$  is a \emph{partially hyperbolic skew product} if the following holds:
\begin{itemize}
\item $f$ preserves fibers;
\item $f$ is a partially hyperbolic diffeomorphism, with splitting $TM =E^s \oplus E^c \oplus E^u$, such that $E^c = \ker D\pi$.
\end{itemize} 

Let $\leb$ be the normalized Lebesgue measure on $M$ and define $\SPV$ to be the set of $C^r$-partially hyperbolic skew products that preserve the Lebesgue measure. In the space $\SPV$ we consider the $C^{\hat{r}}$-topology, for any $\hat{r}\leq r$. 

By Oseledets' theorem (see for instance \cite{BaP01}), for Lebesgue almost every point $x\in M$, the greatest and smallest Lyapunov exponents along the center direction, defined respectively by 
$$\lambda_c^+(x)=\lim_{n\to + \infty} \frac{1}{n} \norm{Df^n(x)|_{E^c_x}}\textrm{ and } \lambda_c^-(x)=-\lim_{n\to +\infty} \frac{1}{n} \norm{Df^{-n}(x)|_{E^c_x}},$$
exist. The conditions in the definition of partially hyperbolic skew product implies that $\det Df(x)|_{E^c_x} = 1$ (see section \ref{sec.preliminaries}). This implies that for almost every point $x\in M$ it is verified that $\lambda_c^-(x) = -\lambda_c^+(x)$. We define the \emph{integrated Lyapunov exponent along the center direction} by
\[
L(f) = \displaystyle \int_M \lambda_c^+(x) d\leb(x).
\]

An important notion in the study of Lyapunov exponents for partially hyperbolic diffeomorphisms is the notion of center bunching, this is used to obtain the existence of linear holonomies, see section \ref{sec.preliminaries} for a precise definition. In this paper we prove the following theorem.

\begin{ltheorem}
\label{thm.theorema}
For any $r>1$, among the volume preserving, $C^r$-partially hyperbolic skew products that are center bunched, there exists a $C^1$-dense and $C^r$-open subset of diffeomorphisms verifying the following: if $f$ belongs to this subset, then $L(f)>0$.
\end{ltheorem}

We say that $f\in \SPV$ is \emph{non-uniformly hyperbolic} if for Lebesgue almost every point it holds that $\lambda^+(x) >0$. From \cite{AV4} (or \cite{HoS17}), it is known that ergodicity is $C^1$-open and $C^r$-dense in the setting of the previous theorem. The next result follows immediately from theorem \ref{thm.theorema}.

\begin{corollary}
In the same setting of theorem \ref{thm.theorema}, there exists a $C^1$-dense and $C^r$-open subset such that any diffeomorphism in this subset is non-uniformly hyperbolic.
\end{corollary} 

With some extra condition (called \emph{pinched}), Marin in \cite{Mar16}, proved the density of non-uniform hyperbolicity for partially hyperbolic maps volume preserving maps with $2$-dimensional center direction, in her argument accessibility and volume preserving are crucial properties because they used the results of \cite{ASV13}. As our argument is not based on \cite{ASV13}, we do not use accessibility and neither all the properties of the volume preserving maps, so we can extend
our results to more general skew products.

Let $\Sigma$ be a compact metric space with no isolated points, let $\sigma:\Sigma \to \Sigma$ be a hyperbolic homeomorphism and $\tilde{\mu}$ be a $\sigma$-invariant measure that has a property called \emph{local product structure} (see section \ref{subsection.hyphomeo} for precise definitions). This property holds for important measures such as the equilibrium states of H\"older potentials (see \cite{Bow75a}).

Let $S$ be a compact, oriented $C^r$-surface. Fix some $\alpha>0$, by abuse of notation let $\leb$ be the normalized Lebesgue measure on $S$ and $ \Dif^r_\leb(S)$ be the space of $C^r$ diffeomorphisms that preserves $\leb$. Given a $(H,\alpha)$-H\"older map from $\Sigma$ to $ \Dif^r_\leb(S)$, $\hx  \mapsto  f_\hx$, by this we mean that 
\[
\displaystyle \d_{C^{\hat{r}}}(f_\hx,f_\hy) \leq H  \d_{\Sigma}(\hx, \hy)^\alpha.
\]
We define the skew product 
\[
\begin{array}{rcc}
f:{\Sigma\times S} &\to & \Sigma \times S\\
(\hx,t) & \mapsto & \quad f(\hx,t) =(\sigma(\hx),f_\hx(t)).
\end{array}
\]
Observe that such a skew product preserves the measure $\mu := \tilde{\mu} \times \leb$. Such a map is called \emph{$C^{r,\alpha}$-skew product} over $\sigma$ that preserves $\mu$. 

From now on fix $C$, $\alpha$  and $r\geq 1+\alpha$, we write $\SP({\Sigma\times S})$ for the space of such skew products over $\sigma$. In this space we consider the $C^{\hat{r}}$-topology, for any $\hat{r}\leq r$ defined as follows: for any two $C^{r,\alpha}$-skew products $f,g\in \SP({\Sigma\times S})$, the $C^{\hat{r}}$ distance between $f$ and $g$ is
\begin{equation}\label{eq.distance} 
\d_{C^{\hat{r}}}(f,g) = \sup_{\hx \in \Sigma} \d_{C^{\hat{r}}(S)}(f_{\hx}, g_{\hx}),
\end{equation}
where $\d_{C^{\hat{r}},\hx}(f_{\hx}, g_{\hx})$ is the $C^{\hat{r}}$ distance on $\Dif^r_\leb(S)$. Keep in mind that $\sigma$ is always fixed.

As our map $f$ is smooth on the fiber direction, we can define the center Lyapunov exponents as 
$$\lambda^+(\hx,t)=\lim_{n\to + \infty} \frac{1}{n} \norm{Df^n_\hx (t)}\textrm{ and } \lambda^-(\hx,t)=-\lim_{n\to +\infty} \frac{1}{n} \norm{Df^{-n}_\hx(t)},$$
where $f^n_\hx=f_{\sigma^{n-1}(\hx)}\circ \dots \circ f_\hx$. This is defined $\mu$-almost everywhere.

Similar to the notion of center bunching, there is a notion of fiber bunching which guarantees the existence of linear holonomies, see section~\ref{sec.preliminaries} for precise definitions. 

\begin{ltheorem}\label{teo.B}
Let $\sigma$ be a hyperbolic homeomorphism and let $\tilde{\mu}$ be a $\sigma$-invariant measure with local product structure. For any $r>1$ and $\alpha>0$, there exists a $C^1$-dense and $C^r$-open subset of $\SP({\Sigma\times S})$ verifying the following: if $f$ belongs to this subset, then $L(f)>0$.
\end{ltheorem}

In Theorem~\ref{teo.B} we need to fix the H\"older constant because, as opposed to the $C^r$ distance in the setting of Theorem~\ref{thm.theorema}, the distance defined in \eqref{eq.distance} does not take in account the relation between the H\"older norm of $\hx\mapsto f_\hx$ and $\hx\mapsto g_\hx$. This hypothesis can be replaced by a finer $C^r$ topology that takes into account the H\"older distance between the maps defining $f$ and $g$.

\begin{remark}
Actually in theorems~\ref{thm.theorema} and \ref{teo.B} the open sets are $C^1$ open sets on subsets of $C^r$ with bounded $C^r$ norm if $r<2$ and bounded $C^2$ norm if $r\geq 2$. 
\end{remark}

A particular case very studied in the literature is called \emph{random product of diffeomorphisms}, see for instance~\cite{KoNa10}, \cite{Bro17}.
This is defined by a set of $C^r$-diffeomorphisms $f_1,\dots,f_d\in \Dif^r_\leb(S)$ and $p_1,\dots,p_d$ positive real numbers such that $p_1+\dots+p_d=1$ 
where the probability $p$ of the diffeomorphism $f_i$ to act on $S$ is $p_i$. Formally this is a skew product of
over the shift map in $\{1,\dots,d\}^\integer$ with Bernoulli probability $P=p^\integer$. With our techniques, we obtain the following theorem.
\begin{ltheorem}\label{teo.C}
Given $d\in \natural$ and some probability on $p$ on $\{1,\dots,d\}$ then in $\Dif_\leb (S)^d$ there exist a
$C^1$ open and dense set such that if $(f_1,\dots,f_d)$ belongs to this set its random product has positive integrated Lyapunov exponents.
\end{ltheorem}
Observe that in this case we actually get a $C^1$ open set.

We can only get $C^1$ density because we use the result of \cite{LY17} where they prove the $C^1$ density of volume preserving diffeomorphisms with positive Lyapunov exponents to find what we call \emph{pinching} points (see definition~\ref{def.pinching}), this result is not known in the $C^r$ topology.

With some information on the periodic diffeomorphism, pinching can be found in some higher regularity, as we explain in section~\ref{ss.cr-density} we have the following result. 

\begin{ltheorem}\label{thm.elliptic}
Let $f$ be as in theorem~\ref{thm.theorema}, \ref{teo.B} or \ref{teo.C} and there exist some periodic fiber $S_\tp$ such that $f_\tp:S_\tp\to S_\tp$ has an elliptic periodic point. Then $f$ is $C^r$-accumulated by open sets with positive integrated Lyapunov exponents. Moreover, in the random product case this sets are $C^1$ open. 
\end{ltheorem}

\subsection*{Acknowledgements}
The author would like to thank Sylvain Crovisier for useful conversations. D.O. and M.P were supported by the ERC project 692925 NUHGD.

\section{Preliminaries and precise statements }\label{sec.preliminaries}
In this section we recall some definitions and give the precise statements of the main theorems. 
\subsection{Disintegration of measures}
Let $M$ be a fiber bundle over $\tM$, $\mu$ be a probability measure on $M$ and $\tmu=\pi_*\mu$,
we say that a family of probability measures $\tx\mapsto \mu_\tx$, defined $\tmu$-almost everywhere, is a disintegration 
of $\mu$ with respect to the fibers if
\begin{itemize}
\item for every $A\subset M$, $\tx\mapsto \mu_\tx(A)$ is measurable,
\item $\mu(A)=\int \mu_\tx(A)d\tmu(\tx)$,
\item $\mu_\tx(\pi^{-1}(\tx))=1$.
\end{itemize}

The partition into fibers verifies a measurability condition and by Rokhlin disintegration theorem (see \cite{FET} chapter $5$ for details) for any probability measure $\mu$ there exists a disintegration of $\mu$ with respect to the fibers. Moreover this disintegration is unique almost everywhere.

\subsection{ Partial hyperbolicity and restatement of theorem \ref{thm.theorema}}
\label{subsection.ph}
A $C^r$-diffeomorphism $f:M \to M$ is partially hyperbolic if the tangent bundle has a $Df$-invariant decomposition $TM = E^s \oplus E^c \oplus E^u$ and there is a riemannian metric such that for any $x\in M $ it holds
\[\arraycolsep=1.4pt\def\arraystretch{1.3}
\begin{array}{l}
\norm{Df(x)|_{E^s_x}} < 1 < m(Df(x)|_{E^u_x}),\\ 
\norm{Df(x)|_{E^s_x}} < m(Df(x)|_{E^c_x}) \leq \norm{Df(x)|_{E^c_x}} < m(Df(x)|_{E^u_x}),
\end{array}
\]
where $m(Df(x)) = \norm{Df(x)^{-1}}^{-1}$ is the co-norm. It is well known that the distribution $E^s$ is uniquely integrable, that is, it exists a unique foliation $\mathcal{F}^s$ tangent to $E^s$, whose leaves are $C^r$ immersed submanifolds. For a point $x\in M$ we denote the leaf of such foliation that contains $x$ by $W^{ss}(x)$ and we call it the \emph{strong stable manifold} of $x$. Similarly, we define the strong unstable manifold of $x$ and denote it by $W^{uu}(x)$. If the subbundle $E^c$ is trivial, then we say that $f$ is \emph{Anosov}. 

Recall that $\SPV$ is the space of partially hyperbolic skew products, for the fiber bundle $M$, with fibers $S$ and base $\tM$. Let $f\in \SPV$, the invariance of the fibers implies that there exists some $C^r$-diffeomorphism $\tf:\tM\to \tM$ such that 
$\tf\circ\pi=\pi\circ f$. Since $\pi$ is a smooth map, $\tmu=\pi_* \leb$ is a smooth volume measure in $\tM$.

Let $x\to \mu^c_x$ be the disintegration of $\leb$ on the fibers $S_x$. Since $M$ is a smooth fiber bundle,
the measure $\mu^c_x$ is a smooth volume measure on $S_x$ and $x\to \mu^c_x$ is continuous in the weak$^*$ topology.
Observe that the invariance of $\leb$ implies that for $\tmu$-almost every $\tx\in \tM$, ${f_\tx}_*\mu^c_\tx=\mu^c_{\tf(\tx)}$, 
then by the continuity of the disintegration, this is actually true for 
every $\tx\in \tM$.

\begin{lemma}\label{l.anosov}
 $\tf:\tM\to \tM$ is an Anosov diffeomorphism.
\end{lemma}
\begin{proof}
Since the directions $E^c$ and $E^s$ have angle bounded away from $0$, there exists a constant $C>0$ such that for every $v\in E^s_x$, it holds 
 $$\frac{1}{C}\norm{v}\leq \norm{D\pi(x)v}\leq C\norm{v}.$$
 
 Let $\tx=\pi(x)$ and define $\tilde{E}^s_\tx=D\pi(x)E^s_x$. Since $D\pi(f(x)) Df(x)=D\tf(\tx) D\pi(x)$, for $\tv\in \tilde{E}^s_\tx$ such that $\tv =D\pi(x)v$, with $v\in E^s_x$, for every $n\in \integer$ we have
 $$\begin{aligned}
 \norm{D\tf^n(\tx)\tv}&=\norm{D\pi(f^n(x)) Df^n(x)v}\\
 &\leq C\norm{Df^n(x)v}\\
 &\leq C^2 \norm{Df^n(x)\mid_{E^s_x}}\norm{\tv}.
 \end{aligned}$$
Analogously, $ C^{-2} \norm{{Df^n(x)\mid_{E^s_x}}^{-1}}^{-1}\norm{\tv}\leq \norm{D\tf^n(\tx)\tv}$, hence, any $\tv\in \tilde{E}^s_\tx$ is contracted exponentially fast when $n\to \infty$ 
 and expanded exponentially fast when $n\to -\infty$. 
  We can also define $\tilde{E}^u_\tx=D\pi(x)E^u_x$, this subspace is contracted when $n\to -\infty$ and expanded when $n\to \infty$. 
  
 Since $T_xM= E_x^s\oplus E_x^u \oplus E_x^c$ and $E^c_x = \ker D\pi(x)$, we conclude that $T_{\tx}\tM = E^s_\tx \oplus E^s_\tx$. The exponential expansion and contraction of these directions implies that this subspaces are uniquely defined, meaning they do not depend on $x\in \pi^{-1}(\tx)$. We can define such a splitting for any $\tx\in \tM$ and conclude that $D\tf(\tx)\tilde{E}^*_\tx=\tilde{E}^*_{\tf(\tx)}$ for $*=u,s$, therefore $\tf$ is Anosov.
 \end{proof}

\begin{definition}[$\alpha$-center bunching]
\label{def.centerbunching}
A $C^1$-partially hyperbolic diffeomorphism $f$ is $\alpha$-center bunched if for every $x\in M$ it holds
\[
\frac{\norm{Df(x)|_{E^c_x}}}{m(Df(x)|_{E^c_x})} \norm{Df(x)|_{E^s_x}}^{\alpha} <1 \textrm{ and } 1< \frac{m(Df(x)|_{E^c_x})}{\norm{Df(x)|_{E^c_x}}} m(Df(x)|_{E^u_x})^{\alpha}.  
\]
\end{definition}

Observe that this condition is $C^1$-open. For $\alpha>0$, we define $CB^{r,\alpha}_{\leb}(M)$ as the $C^1$-open set of $\alpha$-center bunched, $C^r$-diffeomorphisms inside $\SPV$. If $\alpha=1$ we just write $CB^r_{\leb}(M):= CB^{r,1}_{\leb}(M)$. We now state the precise statement of theorem \ref{thm.theorema}

\begin{theorema}
If $r=1+\alpha $, for $\alpha\in (0,1)$, it exists a $C^1$-dense and $C^r$-open subset of $CB^{r,\alpha}_{\leb}(M)$ such that for any diffeomorphism $f$ in this subset it holds that $L(f)>0$. If $r\geq 2$, then the same result holds inside $CB^{r}_{\leb}(M)$.
\end{theorema}

\subsection{Hyperbolic homeomorphisms and restatement of theorem \ref{teo.B}}\label{s.hyphom}
\label{subsection.hyphomeo}
Let $\Sigma$ be a compact metric space with no isolated points. A homeomorphism $\sigma:\Sigma\to \Sigma$ is called \emph{hyperbolic} if for some $\epsilon>0$, for any $\hx\in \Sigma$, there exist \emph{local stable} and \emph{unstable sets} of $\hx$ with respect to $\sigma$ defined by 
\begin{align*}
W^{s}_{loc}\left(\hy\right) & =\lbrace \hx, \d(\sigma^k(\hx),\sigma^k(\hy))<\epsilon\text{ for every }k\geq 0\rbrace \text{ and }\\
W^{u}_{loc}\left(\hy\right) & =\lbrace \hx: \d(\sigma^k(\hx),\sigma^k(\hy))<\epsilon\text{ for every }k\leq 0\rbrace.
\end{align*}
such that there exist $0<\lambda<1$ and $\tau>0$ with the properties
\begin{itemize}
\item[(i)] $\d(\sigma^n(\hy_1),\sigma^n(\hy_2)) \leq \lambda^n \d(\hy_1,\hy_2)$ for any
      $\hy \in \Sigma$, $\hy_1, \hy_2 \in W^s_{loc} (\hy)$ and $n \geq 0$;
\item[(ii)] $\d(\sigma^{-n}(\hy_1), \sigma^{-n}(\hy_2)) \leq \lambda^n \d(\hy_1,\hy_2)$ for any
      $\hy \in \Sigma$, $\hy_1, \hy_2 \in W^u_{loc} (\hy)$ and $n \geq 0$;
\item[(iii)] if $\d(\hx, \hy)\leq\tau$, then $W^u_{loc}(\hx)$ and $W^s_{loc}(\hy)$ intersect in a
      unique point, which is denoted by $[\hx,\hy]$ and depends continuously on $\hx$ and $\hy$.
\end{itemize}
Anosov diffeomorphisms, Markovian shifts, non trivial hyperbolic atractors and horseshoes are examples of hypebolic homeomorphisms.
Property (iii) defines a local product structure of $\Sigma$, this means for every $\hx\in \Sigma$ there exists a neighborhood
$\hx \in\cV\subset \Sigma$ such that $[\cdot,\cdot]:W^s_{loc}(\hx)\times W^u_{loc}(\hx)\to \cV$, $(\hy,\hz)\mapsto [\hy,\hz]$ is a homeomorphism. 

\begin{definition}
\label{def.productstructure}
Let $\hmu$ be a $\sigma$-invariant probability measure, we say that $\hmu$ has \emph{product structure} if locally in the product coordinates we can 
write $\hmu=\rho \mu^s\times \mu^u$,  where $\mu^s$ is a measure on $W^s_{loc}(\hx)$, $\mu^u$ is a measure on $W^u_{loc}(\hx)$ and $\rho$ is a positive measurable function. We also assume that $\hmu$ is fully supported.
\end{definition}

The local product structure property is verified by equilibrium states of H\"older potentials, also for the Lebesgue measure for Anosov diffeomorphisms.
Consider the normalized Lebesgue measure on $S$, which by abuse of notation we write $\leb$, and define $\mu=\hmu\times \leb$. Observe that this is a $f$-invariant probability 
measure for every $f\in \SP({\Sigma\times S})$.

Take $\beta>0$, we say that $f$ is \emph{$\beta$-fiber bunched} if for any $\hx\in \Sigma$ it holds
\begin{equation}
\label{eq.fiberbunchedshift}
\frac{\norm{Df_{\hx}}}{m(Df_{\tx})}\lambda^\beta<1 \textrm{ and } \frac{\norm{Df^{-1}_{\hx}}}{m(Df^{-1}_{\hx})} \lambda^{\beta} <1, 
\end{equation} 
where $\|Df_{\hx}\| = \displaystyle \sup_{t\in S_{\hx}} \{ \|Df_{\hx}(t)\|\}$ and $m(Df_{\hx}) = \displaystyle \sup_{t\in S_{\tx}} \{ \|Df_{\hx}(t)^{-1}\|^{-1}\}$. 
Observe that this condition is $C^1$-open. Define $\mathcal{FB}^{r,\beta}_{\leb}(\Sigma \times S)$ as the set of $\beta$-fiber bunched skew products inside $\SP(\Sigma \times S)$. 

\begin{theoremb}
Let $\sigma:\Sigma\to \Sigma$ be a hyperbolic homeomorphism and $\hmu$ a $\sigma$-invariant probability measure with product structure
let $\mu = \tmu \times leb$. For $r>1$ and $\alpha>0$, there exists a $C^r$-open and $C^1$-dense subset of  $\mathcal{FB}^{r,\alpha^2}_{\leb}({\Sigma\times S})\subset \SP({\Sigma\times S})$ 
such that any $f$ belonging to this subset verifies $L(f,\mu)>0.$
\end{theoremb}

\section{Holonomies and the invariance principle}
\label{s.holonomies}
In this section we are going to define the key concepts of holonomies that we are going to use. We will also introduce the invariance principle, which has been used many times in the study of Lyapunov exponents for cocycles. 
\subsection{Holonomies} 
Recall that by lemma~\ref{l.anosov}, $\tf:\tM\to \tM$ is an Anosov map, in particular it is a hyperbolic homeomorphism.
Therefore, we will define the following concepts in the topological setting. 

For $\hx\in \Sigma$, we denote $S_\hx$ the fiber over $\hx$. In the topological case, as the space is a product the sub-index is not important but we
use it just to stress that the definitions work for smooth fiber bundles.

We say that $f$ admits \emph{$\alpha$-H\"older stable holonomies} if for every $\hy \in W^s_{loc}(\hx)$ 
there exist functions $h^s_{\hx \hy}:S_{\hx}\to S_{\hy}$ such that
\begin{enumerate}
\item[(a)] $h^{s}_{\sigma^j(\hx),\sigma^j(\hy)}=f^j_{\hy} \circ h^s_{\hx,\hy} \circ ({f^j_{\hx}})^{-1}$ for every $j\geq 1$;
\item[(b)] $h^{s}_{\hx,\hx}=\id$ and $h^s_{\hx,\hy}=h^s_{\tz,\hy}\circ h^s_{\hx,\tz}$, for any $z\in W^{s}_{loc}(\hx)$;
\item[(c)] there exists $L>0$ such that $\d_{\Sigma\times S} (h^s_{\hx,\hy}(t),t) \leq L \d_\Sigma(\hx,\hy)^{\alpha}$ for every $t\in S_{\hx}$;
\item[(d)] $\hx,\hy\mapsto h^s_{\hx,\hy}$ is uniformly continuous in $\lbrace \hx,\hy \in \Sigma: \hy \in W^{s}_{loc}(\hx)\rbrace$.
\end{enumerate}

For $p\in {\Sigma\times S}$ define $Df_c(p):=Df_{\pi(p)}(p)$, where $\pi:{\Sigma\times S}\to \Sigma$ is the natural projection.
We define the local strong stable set of $x\in {\Sigma\times S}$ as the set 
\begin{equation}
\label{eq.strongstableset}
W^{ss}_{loc}(x)=\{h^s_{\hy,\hx}(x),\text{ where }\hx=\pi(x)\text{ and }\hy\in W^s_{loc}(\hx)\}. 
\end{equation}

We remark that while the set $W^s_{loc}(\tilde{x})$ is the strong stable set for $\sigma$ contained in the basis $\Sigma$, the set $W^{ss}_{loc}(x)$ is the strong stable set for the skew product $f$, which is contained in $\Sigma \times S$, see figure \ref{fig.strongstable}. We also remark that $\pi(W^{ss}_{loc}(x)) = W^s_{loc}(\tilde{x})$.

\begin{figure}[h]
\centering
\includegraphics[scale= 0.8]{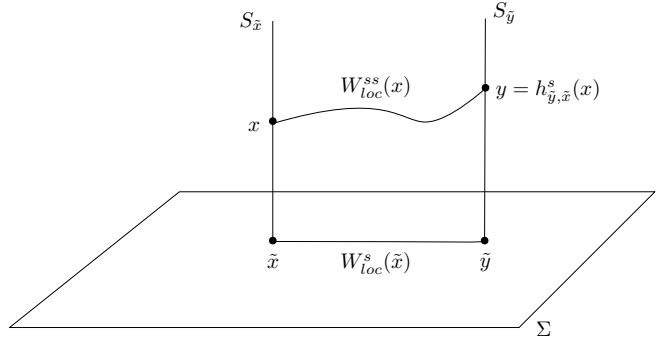}
\caption{The strong stable set}
\label{fig.strongstable}
\end{figure}

Let $TS=\{(\hx,t,v);(\hx,t)\in \Sigma\times S, v\in T_t S_\hx\}$ be the fiber bundle tangent to $S$. For $y\in W^{ss}_{loc}(x)$, we say that $f$ admits \emph{$\alpha$-H\"older linear stable holonomies}, if there exist maps $H^s_{x y}:T_x S_{\hx}\to T_y S_{\hy}$ such that:
\begin{enumerate}
\item[(a)] $H^{s}_{f^j(p),f^j(q)}=Df^j_c(q) \circ H^s_{p,q} \circ (Df^j_c(p))^{-1}$, for every $j\geq 1$;
\item[(b)] $H^{s}_{p,p}=\id$ and $H^s_{p,q}=H^s_{z,q}\circ H^s_{p,z}$, for any $z\in W^{ss}_{loc}(p)$;
\item[(c)] there exists $L>0$ such that $\d_{TS}(H^s_{p,q}(v),v) \leq L \d(p,q)^{\alpha}$;
\item[(d)] $p,q\mapsto H^s_{p,q}$ is uniformly continuous in $\lbrace p,q \in {\Sigma\times S}; q \in W^{ss}_{loc}(p)\rbrace$.
\end{enumerate}

We are going to prove the existence of the holonomies for fiber-bunched skew products of $\SP({\Sigma\times S})$.
\begin{proposition}
\label{prop.holonomyexistence}
If $f$ is $\alpha$-fiber bunched, then for every $\hx\in W^s_{loc}(\hy)$  the limit
$$
\lim_{n\to \infty}{(f^n_\hy)}^{-1}\circ f^n_\hx,
$$
exists and defines an $\alpha$-H\"older stable holonomy.
\end{proposition}
\begin{proof}
For $n$ sufficiently large we can identify $S_{\sigma^n(\hx)}$ with $S_{\sigma^n(\hy)}$ using local charts. Define $h^n={(f^n_\hy)}^{-1}\circ f^n_\hx$. We are going to prove that the maps $h^n:S\to S$ form a Cauchy sequence in the $C^0(S)$-topology. Using $(\ref{eq.fiberbunchedshift})$ and that $f_{\tilde{x}}$ varies $\alpha$-H\"older continuously with the base point, we obtain the following estimate
$$
\begin{aligned}
\d(h^{n+1}(t),h^n(t))&=\d\left((f^{n}_\hy)^{-1}\circ f^{-1}_{\tf^n(\hy)}\circ f_{\tf^n(\hx)}\circ f^n_\hx(t),(f^{n}_\hy)^{-1}\circ f^n_\hx(t)\right)\\
&\leq \sup_{t\in S}\norm{(Df^n_\hy(t))^{-1}}\d_{C^0(S)}\left( f^{-1}_{\tf^n(\hy)}\circ f_{\tf^n(\hx)},\id\right)\\
&\leq \sup_{t\in S}\norm{(Df^n_\hy(t))^{-1}} C \lambda^{n\alpha}\d(\hx,\hy)^\alpha\\
&\leq C\theta^n\d(\hx,\hy)^\alpha, 
\end{aligned}
$$
for some $\theta\in (0,1)$. Hence, $(h^n)_{n\in \natural}$ is a Cauchy sequence and converges uniformly in the $C^0$-topology. Thus, we define $h^s_{\hx,\hy}=\lim_{n\to \infty} h^n$. The properties of the holonomy follow directly from the definition of the limit.
\end{proof}
The hypothesis in proposition \ref{prop.holonomyexistence} could be weakened. Indeed, we do not need the $\alpha$-fiber bunching condition, but something weaker that can be seen as a type of ``dominated splitting'' condition (the contraction on the basis is stronger than the contractions on the fibers). Similar considerations also hold for the unstable holonomies.

\begin{remark}
 For the smooth case the holonomy can also be defined by the strong stable foliation restricted to a center stable manifold,
 specifically $h_{\tx \ty}(t)=W^{ss}_{loc}(t)\cap S_\ty$.
\end{remark}

To define the $\alpha$-H\"older linear stable holonomy we first need to find the contraction rate of $f$ in the strong stable set.

Take $(\hx,t)$ and $(\hy,h^s_{\hx \hy}(t))$ in the same strong stable set, then
$$
\begin{aligned}
\d_S(f^n(h^s_{\hx,\hy}(t)),f^n(t))&=\d_S(h^s_{\tf^n(\hx)\tf^n(\hy)} f^n(t),f^n(t))\\
&\leq L\d_{\Sigma}(\tf^n(\hx),\tf^n(\hy))^\alpha \\
&\leq L \lambda^\alpha \d_{\Sigma}(\hx,\hy)^\alpha \\
&\leq L \lambda^\alpha \d_{\Sigma\times S}((\hx,t),(\hy,h^s_{\hx \hy}(t)))^\alpha
\end{aligned}
$$

so the contraction rate is at least $\lambda^\alpha$.

\begin{proposition}
\label{prop.existencelinearholonomies}
If $f$ is $\alpha^2$-fiber bunched, then for every $x\in W^{ss}_{loc}(y)$ the limit 
$$
\lim_{n\to \infty}{(Df^n_c(y))}^{-1}\circ Df^n_c(x),
$$
exists and defines an $\alpha$-H\"older linear stable holonomy.
\end{proposition}

\begin{proof}
For $n$ sufficiently large we can identify $T_{f^n(x)}S_{\sigma^n(\hx)}$ with $T_{f^n(y)}S_{\sigma^n(\hy)}$ using local charts, also observe that  as $r-1\geq\alpha$ there exists $C>0$ such that
$$
\begin{aligned}
\norm{(Df^n_\ty(t))-(Df^n_\tx(t'))}&\leq H \d_\Sigma(\tx,\ty)^\alpha+\norm{f}_{C^{\min(r,2)}} \d_S(t,t')^{\min(r-1,1)}\\
&\leq C \d_{\Sigma\times S}((\ty,t)(\tx,t'))^\alpha.
\end{aligned}
$$

Define $H^n={(Df^n_c(y))}^{-1}\circ Df^n_c(x)$. We are going to prove that $(H^n)_{n\in \natural}$ is a Cauchy sequence.
$$
\begin{aligned}
\norm{H^{n+1}-H^n}&=\norm{(Df_c^{n+1}(y))^{-1}\circ Df_c^{n+1}(x)-(Df_x^{n}(y))^{-1}\circ Df_c^n(y)}\\
&\leq \norm{(Df_c^n(y))^{-1}}\norm{ (Df_c(f^n(y)))^{-1} \circ Df_c(f^n(x))-\id}\norm{(Df_c^n(x))}\\
&\leq C \norm{(Df_c^n(y))^{-1}}\norm{(Df_c^n(x))} \d(f^n(x),f^n(y))^\alpha\\
&\leq \norm{(Df_c^n(y))^{-1}}\norm{(Df_c^n(x))} \lambda^{\alpha^2 n}\d(x,y)^\alpha.
\end{aligned}
$$
Observe that the $\alpha^2$-fiber bunching condition is open, in particular, there exist $\theta \in (0,1)$ and $\delta>0$ such that if $\d(x',y')<\delta$, then 
$$\norm{(Df_c(y'))^{-1}}\norm{(Df_c(x'))} \lambda^{\alpha^2}<\theta.$$
Therefore, for $j$ sufficiently large
$$\norm{Df_c(f^j(y))^{-1}}\norm{Df_c(f^j(x))} \lambda^{\alpha^2}<\theta.$$
Thus, 
 $$
 \norm{H^{n+1}-H^n}\leq C\theta^n \d(x,y)^\alpha.
 $$
The sequence $(H^n)_{n\in \natural}$ is a Cauchy sequence and converges uniformly. Define $H^s_{x,y}=\lim_{n\to \infty} H^n$. The properties of the holonomy follow directly from the definition of the limit.
\end{proof}
\begin{remark}
 For the $C^{1+\alpha}$-diffeomorphism case as the $\alpha$-center bunching condition takes in account the contraction over the strong stable direction
the argument works with $\alpha$-center bunching instead of $\alpha^2$-fiber bunched.
\end{remark}

\begin{remark}
In the $C^{1+\alpha}$-diffeomorphism case, the existence of the linear holonomies, given by proposition \ref{prop.existencelinearholonomies}, does not imply that the holonomy maps $h^s$ are $C^1$. Indeed, to conclude that the maps $h^s$ are $C^1$, one needs to prove that the sequence $(h^n)_{n\in \natural}$ is Cauchy for the $C^1$-topology, where this sequence was defined in the proof of proposition \ref{prop.holonomyexistence}. The calculation becomes more delicate and in fact one needs a stronger bunching condition, see \cite{Bro16}.
\end{remark}
\begin{remark}\label{rem.continuity.holon}
Observe that the convergence in propositions \ref{prop.holonomyexistence} and \ref{prop.existencelinearholonomies} is uniform in sets with bounded H\"older constant. 

This implies that the holonomies $h^s$ and $h^u$ vary continuously with $f$ with respect to the $C^0$ topology in subsets with bounded H\"older constant, and the linear holonomies $H^s$ and $H^u$ vary continuously with $f$ with respect to the $C^1$ topology on subsets on subsets with bounded H\"older constant of $Df_c$. 

In particular in the random product case the holomies varies continuously with $f_1,\cdots,f_k$ in the $C^1$ topology.
\end{remark}
Analogously we define the unstable and linear unstable holonomies as the stable and linear stable holonomies for $f^{-1}$.

\subsection{The invariance principle and criterion for the existence of positive exponents}
Let $f\in \mathcal{FB}^{r, \alpha^2}_{leb}(\Sigma \times S)$ be a skew product over a  hyperbolic homeomorphism $\sigma : \Sigma \to \Sigma$. In \cite{Bow70}, Bowen proved that a hyperbolic homeomorphism is semi-conjugated to a subshift of finite type. In particular, the set of periodic points of $\sigma$ is non empty. Also using the symbolic dynamics associated to $\sigma$, it holds that for any periodic point $\hp$ it exists $\hz \in \Sigma$ such that $\hz \in W^u_{loc}(\hp) \cap \sigma^{-i}(W^s_{loc}(\hp)) -\{\hp\}$, for some $i\in \natural$.

\begin{definition}[Pinching]
\label{def.pinching}
We say that the cocycle $(f,Df_c)$ is \emph{pinching} if there exist a periodic point $\hp\in \Sigma$, for $\sigma$, such that the cocycle $(f^\kappa_{\hp},Df^\kappa_{\hp})$ verifies
\[
\int_{S_{\hp}} \lambda^+(Df^\kappa_{\hp},t) d\mu^c_{\hp} (t)>0,
\]
where $\kappa$ is the period of $\hp$.
\end{definition}

Let $\tp$ be a periodic point for $\sigma$ of period $\kappa$. Consider $NUH_{\hp} \subset S_{\hp}$ to be the set of points $t$ inside $S_{\hp}$ such that $\lambda^+(Df^\kappa_{\hp}, t )>0$. Let $PS_{\hp}$ be the projectivization of the tangent bundle $TS_{\hp}$. By Oseledets theorem, on $NUH_{\hp}$ there exists a measurable function $t \mapsto  (e^u(t),e^s(t)) \in PS_{\hp} \times PS_{\hp},$ where $e^u(t)$ is the Oseledets space corresponding to the positive Lyapunov exponent and $e^s(t)$ is the Oseledets space corresponding to the negative exponent.

\begin{definition}[Twisting]
\label{def.twisting}
Let $(f, Df_c)$ be a pinching cocycle for the periodic point $\hp\in \Sigma$. We say that the cocycle is \emph{twisting} (see figure~\ref{fig.twiting}) if there exist $j\in \natural$, $\hz\in \tilde{W}^u_{loc}(\hp)\cap \left(\tf^{-i}(\tilde{W}^s_{loc}(\hp))-\{\hp\}\right)$ and a set $K\subset NUH_{\tp}$ such that $\mu^c_\hp(K)>0$ and $K$ verifies the following: 
 $$(H_t)_*^j\left(\lbrace e^+(t),e^-(t) \rbrace\right) \cap \lbrace e^+(h^j(t)),e^-(h^j(t)) \rbrace=\emptyset,$$
where 
\begin{itemize}
\item[--]$h:S_{\hp}\to S_{\hp}$ is the composition $h^s_{\sigma^i(\hz),\hp}\circ f^i_\hz \circ h^u_{\hp,\hz}$;
\item[--] $H_t:T_tS_{\hp}\to T_{h(t)}S_{\hp}$ is the composition $H^s_{f^i(t_{\hz}),h(t)}\circ Df^i_\hz(t_{\hz})\circ H^u_{t,t_{\hz}}$, where $t_{\hz}=h^u_{\hp,\hz}(t)$;
\item[--] $(H_t)^j=H_{h^{j-1}(t)}\circ\dots\circ H_{t}$;
\item[--]$(H_t)_*^j$ is the action induced by $H_t^j$ on the projective bundle $PS_{\hp}$.
\end{itemize} 
\end{definition}

\begin{figure}[ht]
    \centering
    \includegraphics[scale=0.4]{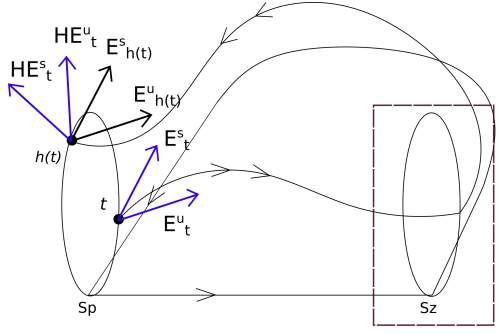}
    \caption{Twisting}
\label{fig.twiting}
   \end{figure}

These two definitions combined with the invariance principle gives us a criterion (theorem~\ref{th.criterionexponents}) for the existence of positive exponents for a open set of $\SP(\Sigma\times S)$.
Let us explain how this follows.

Recall that $TS=\{(\hx,t,v);(\hx,t)\in \Sigma\times S, v\in T_t S_\hx\}$ is the fiber bundle tangent to $S$, the derivative cocycle $Df_c:TS\to TS$  induce an action on the projective fiber bundle $PS$, which we will denote it by $F:PS\to PS$. 

For a probability measure $m$ on $PS$, that projects on $\mu$, we write $m_x$ the disintegration of $m$ with respect to the projective fibers $P_x S_\hx$.

A $F$-invariant measure $m$ is an \emph{$u$-state} if there exists a set of full $\mu$-measure, $M'\subset M$, such that for every $x,y\in M$, with $y\in W^{uu}(x)$, we have that $m_y={H^u_{x, y}}_*m_x$. One defines a $s$-state analogously, replacing the roles of unstable by stable holonomies. If $m$ is both $s$ and $u$ state, we call it a $su$-state. Using results from \cite{Pol16}, we have the following proposition.
\begin{proposition}\label{p.p+t}
If $(f,Df_c)$ pinching and twisting then it does not admit a $su$-state projecting on $\mu$.
\end{proposition} 
\begin{proof}
Suppose that there exists $m$ a $F$-invariant measure that is a $su$-state. Observe that the holonomies $h$ preserves volume on $S$.

By Proposition~7.1 of \cite{Pol16} there exist a disintegration that is $su/c$-invariant, this means that for every $\hx,\hy\in \Sigma$ in the same stable set for $\mu^c_\hx$-almost every $t\in S_\hx$
\begin{equation}\label{eq.su/c}
(H^s_{t,h^s_{\hx,\hy}(t)})_*m_t=m_{h^s_{\hx,\hy}(t)}
\end{equation}
and the same property changing stable by unstable.

The pinching condition implies that for $\mu^c_\hp$-almost every $t\in NUH_\hp$, $m_t=a(t) \delta_{e^u(t)} +b(t) \delta_{e^s(t)}$, where $a(t),b(t)\in \real^+$ with $a(t)+b(t)=1$. Now \eqref{eq.su/c} implies that for $\mu^c_\hp$-almost every $t\in NUH_\hp$ and for every $j\in \integer$ 
$$a(t) \delta_{H_t^j e^u(t)} +b(t) \delta_{ H_t^j e^s(t)}=a(h^j(t)) \delta_{e^u(h^j(t))} +b(h^j(t)) \delta_{e^s(h^j(t))}$$
this contradicts the twisting condition.
\end{proof}

\begin{lemma}\label{lem.closedu}
If $m^k$ are $u$-states for $F_k$, that projects to $\mu$ such that $F_k \to F$ and $m^k\to m$ in the weak*-topology, then $m$ is an $u$-state.
\end{lemma}
The proof of this lemma is very technical and also of independent interest. We prove it in more generality in section section~\ref{s.u-closed}, see theorem \ref{t.closedu}.

\begin{theorem}
\label{th.criterionexponents}
Let $\sigma: \Sigma \to \Sigma$ be a hyperbolic homeomorphism, $\hat{\mu}$ be a $\sigma$-invariant measure with product structure and let $f\in \mathcal{FB}^{1+\alpha,\alpha^2}_{leb}(\Sigma\times S)$. If $(f,Df_c)$ is pinching and twisting then $L(f,\mu)>0$, for $\mu = \hat{\mu} \times \leb$.  Moreover, there exist a neighborhood of $f$ in the $C^r$-topology with positive integrated Lyapunov exponents. 
\end{theorem}
\begin{proof}
By \cite{ASV13} if $\lambda^c(Df)=0$ every $F$-invariant measure is an $su$-state, so using proposition~\ref{p.p+t} we prove that $L(f,\mu)>0$.

For the second part assume that there exist $f_k\to f$ with $L(f_k,\mu)=0$, then every $F_k$-invariant measure $m^k$ is an $su$-state.
Take a sub-sequence of $m^k$ that converges to some $m$. By Lemma~\ref{lem.closedu} $m$ is a $su$-state, which is a contradiction.
\end{proof}

We remark that if $r\geq 2$ we can take $\alpha =1$ in the previous statement. We also remark that the same theorem is true if $f \in CB^{r,\alpha}_{\leb}(M)$.
\section{The proof of theorem A}
The main difference between the case when $r=1+\alpha$ and when $r\geq 2$ is the existence of the linear holonomies, which was treated in section \ref{s.holonomies}. From now on the proof is the same for both cases. So we suppose that $r\geq 2$.

Let $M$ be a fiber bundle over $\tM$, with fiber $S$, and take $f\in CB^{r}_{\leb}(M)$. By lemma \ref{l.anosov}, $f$ induces a diffeomorphism $\tf: \tM \to \tM$ which is Anosov. Moreover, $\tf$ preserves $\tilde{\mu} = \pi_*leb$. As in the previous section, we write $Df_c = Df|_{E^c}$ and recall that $\mu^c_{\tx}$ is the disintegration of the Lebesgue measure on the fiber $S_{\tx}$. Since $M$ is a smooth bundle, the measure $\mu^c_{\tx}$ is just the Lebesgue measure on $S_{\tx}$. 

In the following lemma we show how to perturb the cocycle $(f,Df_c)$, which is pinching, to obtain a cocycle that is also twisting. 

\begin{lemma}
\label{l.twisting.dense}
Let $(f,Df_c)$ be a pinching cocycle for the periodic point $\tp$. Then there exists $g$ arbitrarily $C^r$-close to $f$ such that the cocycle $(g,Dg_c)$ is pinching and twisting. Moreover, $g^\kappa|_{S_\tp}=f^\kappa|_{S_\tp}$, where $\kappa\in \natural$ is the period of the periodic point $\tp$ for $\tf$.  
\end{lemma}
\begin{proof}
Let $K\subset S_{\tp}$ be a compact set with the following properties:
\begin{itemize}
\item $\mu^c_\tp(K)>0$;
\item every point $t\in K$ has one positive and one negative Lyapunov exponent for the cocycle $(f^{\kappa}, Df_c)$;
\item the map $t\mapsto (e^u(t),e^s(t))$ is continuous on $K$, where $e^u(t)$ and $e^s(t)$ are the Osedelets spaces. 
\end{itemize}
Such compact sets always exist, see section $4.2.1$ in \cite{BaP01}. Fix $\tz\in W^u_{loc}(\tp)$ such that $\tf^i(\tz)\in W^s_{loc}(\tp)$, with $i>0$. Write $h=h^s_{\tf^i(\tz),\tp}\circ f_{\tz}^i\circ h^u_{\tp,\tz}$. By the formula given in proposition \ref{prop.holonomyexistence}, we have that $h^u_{\tp,\tz}$ only depends on $f^n(\tz)$ and $f^n(\tp)$, for $n\leq 0$, and $h^s_{\tf^i(\tz),\tp}$ only depends on $f^{n+i}(\tz)$ and $f^n(\tp)$, for $n\geq 0$. 

Take $t$ a density point of $K$ and $j_t\in \natural$ such that $h^{j_t}(t)$ is also a density point of $K$, this is can be done because $h$ preserves $\mu^c_\tp$ and as a consequence of Poincar\'e recurrence theorem (Referencia). We can take $j_t$ to be the smallest natural number that verifies this condition for $t$. Let $t_{\tz}=h^u_{\tp,\tz}(t)$ and let $H_t=H^s_{f^i(t_{\tz}),h(t)}\circ Df_{\tz}^i(t_\tz)\circ H^u_{t,t_{\tz}}$. For $H_t$, the map $H^u_{t,t_{\tz}}$ only depends on $f^n(t)$ and $f^n(\tz)$, for $n\leq 0$, and $H^s_{f^i(t_{\tz}),h(t)}$ only depends on $f^n(h(t))$ and $f^n(f^i(\tz))$, for $n\geq 0$. 	

If the twisting condition did not hold, then for almost every $t'\in K$ we would have
\begin{equation}\label{eq.condt}
(H_{t'})^{j_{t'}}\left(\lbrace e^u(t'),e^s(t') \rbrace\right) \cap \lbrace e^u(h^{j_{t'}}(t')),e^s(h^{j_{t'}}({t'})) \rbrace\neq\emptyset,
\end{equation}
in particular, we can assume that this holds for $t$.

Take $V\subset  S_{\tz}$ a neighborhood of $t_{\tz}$ in $S_{\tz}$ small enough such that the sets $V$, $\cdots$, $h^u_{\tp, \tz} \circ h^{j-1}(V)$ are pairwise disjoint. Also fix a small open neighborhood of $\tz$, $\tilde{U}\subset \tM$, such that the sets $\tilde{U}$, $\cdots$, $\tf^{i-1}(\tilde{U})$ are pairwise disjoint. Consider the neighborhood of $\tz$ given by $\tilde{U} \times V$.

Let $D \subset \real^2$ be the unitary disc and let $U=B^{d-2}(0,1) \subset \real^{d-2}$, where $d$ is the dimension of $M$ and $B^{d-2}(0,1)$ is the unitary ball on $\real^{d-2}$. Take a smooth parametrization $\Phi:U\times D\to M$ such that $\Phi(U\times D) \subset \tilde{U}\times V$ and for each $y\in U$ we have that $\Phi(\{y\}\times D) \subset \pi^{-1}(\pi(\Phi(\{y\}\times D))$. We also take $\Phi$ verifying $\Phi(0)=t_{\tz}$ and $\Phi^{-1}_* \mu\mid _{V \times \tilde{U}}$ is the standard volume in $U\times D$.

Using polar coordinates on $D$, we consider the vector field $X(r',\theta) = r' \frac{\partial}{\partial \theta}$. Let $\rho: [0,1] \to [0,1]$ be a smooth bump function that verifies, $\rho(r') = 1$ if $r'\in [0,\frac{1}{3}]$, and $\rho(r') = 0$ if $r'\in [\frac{2}{3}, 1]$. Using coordinates $(y,r',\theta)$ on $U\times D$, where $y\in U$, we define the vector field $\hat{X}$ on $U \times D$ by
\begin{equation}
\label{eq.vectorfield}
\hat{X}(y,r',\theta) = \left(0,\rho(\|y\|) \rho(r')r' \frac{\partial}{\partial \theta}\right).
\end{equation}

For $T\in \real$, let $\phi_T$ be the flow generated by $\hat{X}$. Observe that $\phi_T=\id$ in a neighborhood of the boundary of $U\times D$, also that $\phi_T(0)=0$ and $D\phi_T(0)=id \times R_T$, where $R_T$ is the rotation counterclockwise of angle $T$. Using the parametrization $\Phi$ we define a flow $\phi'_T : M \to M$ as follows: if $q\notin  \tilde{U} \times V$, then $\phi'_T(q) = \id$. If $q\in V\times \tilde{U}$, then
\[
\phi'_T(q) = \Phi \circ \phi_T \circ \Phi^{-1}(q).
\] 
Take $g_T = f\circ \phi'_T$. Our perturbation does not affect the orbit of $t_{\tz}$, this implies that $h_T^i(t)$ is not affected for $i=0,\dots,j_t$, in particular, $h_T^{j_t}(t) = h^{j_t}(t)$ and the linear holonomy is given by 
\[
H^T_{t}=H^s_{f^i(t_{\tz}),h(t)}\circ Df_{\tz}^i(s)\circ D\Phi(0) \circ R_t \circ \left(D\Phi(0)\right)^{-1}\circ H^u_{t,t_{tz}}.
\]
Furthermore, $\left(H^T_{h(t)}\right)^{j_t-1}=\left(H_{h(t)}\right)^{j_t-1}$. 

For $T>0$ small, $g_T$ is $C^r$-close to $f$ and it holds that 
$$
\left(H^T_t\right)^{j_t}(\{e^u(t),e^s(t)\}) \cap \{e^u(h^{j_t}(t)),e^s(h^{j_t}(t))\} = \emptyset.
$$
Observe that $g_T^{\kappa}|_{S_{\tp}} = f^{\kappa}|_{S_{\tp}}$. Since $h^{j_t}(t)$ is a density point of $K$ and for the points of $K$ the Oseledets splitting varies continuously, we conclude that $g_T$ is twisting.
\end{proof}

As a consequence of this lemma and theorem \ref{th.criterionexponents}, we have the following theorem.

\begin{theorem}\label{t.pin.accum}
Let $f\in CB^{r}_{\leb}(M)$ be a diffeomorphism such that $(f,Df_c)$ is pinching. Then arbitrarily $C^r$-close to $f$, it exist $C^r$-open sets, inside $CB^{r}_{\leb}(M)$, of diffeomorphisms with positive integrated center Lyapunov exponent. 
\end{theorem}
\begin{proof}
By lemma~\ref{l.twisting.dense}, arbitrarily $C^r$-close to $f$ there exists a diffeomorphism $g$ which is pinching and twisting. By theorem~\ref{th.criterionexponents}, we conclude that $L(g,\mu)>0$. By lemma \ref{lem.closedu}, we obtain that $L(.,\mu)$ is positive in a $C^r$-neighborhood of $g$, since otherwise $g$ would be accumulated by diffeomorphisms admiting a $su$-state, which would imply that $g$ has a $su$-state. 
\end{proof}
\begin{remark}
Using theorem~\ref{t.closedu}, we can conclude that the $C^r$-open sets in the statement of theorem \ref{t.pin.accum} is among the partially hyperbolic volume preserving $C^r$-diffeomorphisms, not necessarily skew products. 
\end{remark}

We will need the following theorem.

\begin{theorem}[Theorem $1.2$ in \cite{LY17}]
\label{thm.ly17}
Let $f$ be a volume preserving, $C^r$-diffeomorphism of a compact manifold $M$ of dimension $d\geq 2$. Then arbitrarily $C^1$-close to $f$ it exists a volume preserving $C^r$-diffeomorphism $g$ without zero exponents on a set of positive measure.
\end{theorem} 

In our scenario, we want to use a fibered version of theorem \ref{thm.ly17}. For that we need the following lemma.
\begin{lemma}
\label{lem.smoothpath}
Let $\mathrm{Diff}^r_{\leb}(S)$ be the space of $C^r$-diffeomorphisms of $S$ that preserve the Lebesgue measure. If $g\in \mathrm{Diff}^r_{\leb}(S)$ is sufficiently $C^1$-close to the identity, then it exists a smooth path $(g_t)_{t\in[0,1]} \subset \mathrm{Diff}^r_{\leb}(S)$ connecting $g$ to the identity such that for any $t\in [0,1]$, $g_t$ is $C^1$-close to the identity. 
\end{lemma}
\begin{proof}
This proof uses several basic tools from symplectic geometry which we will not define, but we refer the reader to \cite{Can01}, chapters $3$ and $9$, for all the definitions and results that we use here. 

Let $\omega$ be the volume form that generates $\leb$ on $S$. Since $S$ is a surface, we have that $(S,\omega)$ is a symplectic manifold and the volume preserving diffeomorphisms are the symplectomorphisms of this manifold.

We consider two other symplectic manifolds $(S\times S, \pi_1^* \omega - \pi_2^*\omega)$, where $\pi_i$ is the projection on the $i$-th coordinate, and the cotangent bundle $(T^*S,\omega_0)$ with the canonical symplectic form. Let $\Delta$ be the diagonal on $S\times S$. It is well known that $\Delta$ is a Lagrangian submanifold of $S\times S$ and that the zero section is a Lagrangian submanifold in $T^*S$. By Weinstein's tubular neighborhood theorem (chapter 9 of \cite{Can01}), it exists a smooth symplectomorphism $\varphi:U_1 \to U_2$ taking $\Delta$ to the zero section, where $U_1$ is a neighborhood of $\Delta$ and $U_2$ is a neighborhood of the zero section. 

If $g$ is a $C^r$-symplectomorphism sufficiently $C^1$-close to the identity, then $\mathrm{graph}(g)$ is contained in $U_1$. Furthermore, its graph is a $C^r$-Lagrangian submanifold. Using $\varphi$, $\mathrm{graph}(g)$ is identified with a $C^r$-Lagrangian submanifold $\Gamma_g$ of $U_2$. Since the zero section is transverse to the fibers in $T^*S$ and $\Gamma_g$ is $C^1$-close to the zero section, we have that $\Gamma_g$ is also transverse to the fibers in $T^*S$. Thus, it exists a $1$-form $\nu$ with regularity $C^r$ such that $\Gamma_g= \mathrm{graph}(\nu)$.

It is true that the graph of a $1$-form $\nu$ is Lagrangian if and only if $d\nu=0$. Consider the smooth family of $1$-forms $\nu_t = (1-t) \nu$, for $t\in [0,1]$. Observe that $d\nu_t=0$ and that $\nu_t$ is a path connecting $\nu$ to the zero form. Hence, the family $\left(\mathrm{graph}(\nu_t)\right)_{t\in [0,1]}$ is a smooth family of $C^r$-Lagrangian submanifolds which are $C^1$-close to the zero section. Using $\varphi$ and a transversality argument again, we obtain that each of these graphs defines a $C^r$-diffeomorphism $g_t: S\to S$ which is $C^1$-close to the identity. Furthermore, since its graphs are Lagrangian submanifolds, it holds that $g_t$ is a symplectomorphism for each $t\in [0,1]$. Thus, we have obtained a smooth family of symplectomorphisms $(g_t)_{t\in [0,1]}$, verifying the conditions of the lemma.  
\end{proof}

\begin{proof}[Conclusion of the proof of theorem A]
We claim that the pinching condition is $C^1$-dense inside $CB^{r,\alpha}_{\leb}(M)$. Indeed, let $f\in CB^{r,\alpha}_{\leb}(M)$ and let $\tp\in \tM$ be a periodic point for $\tf$. For simplicity, suppose that $\tp$ is a fixed point and fix a small open neighborhood $\tilde{U}\subset \tM$ of $\tp$ such that in a trivialization chart of the bundle $M$, a neighborhood of $S_{\tp}$ is given by $\tilde{U} \times S$.  

Consider $f_{\tp} = f|_{S_{\tp}}$. By theorem \ref{thm.ly17}, it exists $g:S \to S$ volume preserving which is $C^1$-close to $f_{\tp}$. Take $G=f^{-1}_{\tp} \circ g$ and observe that $G$ is $C^1$-close to the identity. Let $(G_t)_{t\in [0,1]}$ be the smooth path connecting $h$ to the identity given by lemma \ref{lem.smoothpath}. Fix $\rho>0$ small enough such that the ball, $\tilde{B}(\tp, \rho)$, of radius $\rho$ and center $\tp$ is contained in $\tilde{U}$. We define the fibered diffeomorphism
\[
\begin{array}{rcl}
\mathcal{G}: \tilde{B}(\tp,\rho) \times S &\to & \tilde{B}(\tp, \rho) \times S\\
(\tilde{q}, t) & \mapsto & (\tilde{q}, G_{\frac{d(\tp, \tq)}{\rho}}(t)).
\end{array}
\]
We can extend $\mathcal{G}$ to a diffemorphism which is the identity outside $\tilde{U} \times S$ and inside this neighborhood coincides with $\mathcal{G}$. We also denote this diffeomorphism by $\mathcal{G}$. Since $G_t$ is $C^1$-close to the identity, we have that $\mathcal{G}$ is also $C^1$-close to the identity. Also observe that $\mathcal{G} \circ f|_{S_{\tp}} = g$, which has positive integrated Lyapunov exponent. Hence, $\mathcal{G} \circ f$ is pinching, which proves the claim.

Since the pinching condition is $C^1$-dense, using theorem \ref{t.pin.accum} we finish the proof of theorem A.
\end{proof}

\section{The proof of theorem B}

Let $\sigma:\Sigma \to \Sigma$ be a hyperbolic homeomorphism and $\mathcal{FB}^{r,\alpha^2}_{\leb}({\Sigma\times S})$ be the set of fiber bunched skew products defined in section \ref{sec.preliminaries}. By the results in section \ref{s.holonomies} the linear and non-linear holonomies are well defined inside $\mathcal{FB}^{r,\alpha^2}_{\leb}({\Sigma\times S})$. Similar to lemma \ref{l.twisting.dense}, but in the scenario of theorem $B$, we have the following lemma.
\begin{lemma}
\label{l.twisting.dense.homeo}
If $f\in \mathcal{FB}^{r,\alpha^2}_{\leb}(\Sigma \times S)$ is a fiber-bunched cocycle such that $(f,Df_c)$ is pinching for the periodic point $\hp$, then it exists $g\in \mathcal{FB}^{r,\alpha^2}_{\leb}(\Sigma \times S)$ arbitrarily $C^r$-close to $f$ such that the cocycle $(g,Dg_c)$ is pinching and twisting. Moreover, $g^\kappa_\hp=f^\kappa_\hp$, where $\kappa\in \natural$ is the period of the periodic point $\hp$ for $\sigma$.  
\end{lemma}
\begin{proof}
The proof follows the same steps as the proof of lemma \ref{l.twisting.dense}, the only difference is in the construction of the vector field $\hat{X}$ defined in (\ref{eq.vectorfield}). Let us explain how to adapt the construction of such vector field. Let $K \subset \{\tp\}\times S$, $t\in K$, $\tz \in W^s_{loc}(\tp)$, $t_{\tz} = h^u_{\tp, \tz}(t)$, $V\subset \{\tz\} \times S$ and $\tilde{U} \subset \Sigma$ be as in the proof of lemma \ref{l.twisting.dense}.

 Write $t_{\tz} = (t_1, t_2)$. Consider $V' \subset V$ a small neighborhood of $t_2$ inside $V$ and $\varphi: D \to V'$ be a $C^r$-parametrization from the unit disc $D$ onto  $V'$ such that $\varphi(0) = t_2$. Let $U = B(t_1, \delta)$ be the ball inside $\Sigma$ centered in $t_1$ with radius $\delta>0$ small enough such that $U\subset \tilde{U}$. Consider $\rho_1 : U \to [0,1]$ defined by $\rho_1(\tilde{x}) = \frac{d(\tilde{x}, \partial U)}{\delta}$, where $\partial U$ is the boundary of $U$. We remark that the function $\rho_1$ is Lipschitz continuous. Take $\rho_2:[0,1]\to [0,1]$ a smooth bump function as in the proof of lemma \ref{l.twisting.dense} and consider $X(r',\theta)= r'\frac{\partial}{\partial \theta}$ to be the vector field on $D$ using polar coordinates. On $U\times D$ consider the vector field defined on each fiber by
 \[
 \hat{X}(y,r',\theta) = \rho_1(y)\rho_2(r')X(r', \theta)\in \{y\} \times T_{(r',\theta)} D.  
 \]
 The vector field $\hat{X}$ induces a fibered flow, which we denote it by $\phi_t$. Let $\Phi: U\times D \to U\times V'$ be defined by $\Phi = Id \times \varphi$ and for each $t\in \real$ we define $\phi'_t = \Phi \circ \phi_t \circ \Phi^{-1}$. In the same way as in the proof of lemma \ref{l.twisting.dense} we have that $\phi'_t$ extends to a homeomorphism $g_t: \Sigma \times S \to \Sigma \times S$, which is $C^r$ on the fibers. Furthermore, $Dg_t(t_{\tz}) = D\varphi(0) R_t (D\varphi(0))^{-1}$. The rest of the proof is the same as before.       
\end{proof}

Using lemma \ref{l.twisting.dense.homeo}, we can prove in the same way as in theorem \ref{t.pin.accum} the following theorem.

\begin{theorem}\label{t.pin.accum.homeo}
Let $f\in \mathcal{FB}^{r,\alpha^2}_{\leb}(\Sigma \times S)$ be a skew product such that $(f,Df_c)$ is pinching. Then arbitrarily $C^r$-close to $f$, it exist $C^r$-open sets of skew products with positive integrated center Lyapunov exponent. 
\end{theorem}

Using theorems \ref{thm.ly17}, \ref{t.pin.accum.homeo} and lemma \ref{lem.smoothpath}, and following the same constructions in the conclusion of the proof of theorem $A$, we can conclude the proof of theorem $B$.

\section{Proof of theorem $C$}
Recall that $p$ is a probability measure on $\{1, \cdots, d\}$, for some $d\in \natural$, and given $(f_1, \cdots, f_d) \in \mathrm{Diff}^r_{\leb}(S)^d$ we are interested in studying the random product generated by this set with distribution $p$. This random product can be seen as the skew product
$$
\begin{array}{rcl}
f:\{1, \cdots, d\}^{\mathbb{Z}} \times S & \to &\{1,\cdots, d\}^{\mathbb{Z}} \times S\\
((x_j)_{j\in \mathbb{Z}}, t) & \mapsto & ((x_{j+1})_{j\in \mathbb{Z}}, f_{x_0}(t)),
\end{array}
$$
with the measure $\mu = p^{\mathbb{Z}} \times \leb$. The proof of theorem $C$ follows the same lines as theorems $A$ and $B$, but the perturbations are simpler. 

Observe that in this scenario, the local linear and non-linear holonomies are just the identity maps. This follows because the skew-product is constant on open sets of $\{1, \cdots, d\}^{\mathbb{Z}}$. 

We say that a random product is \emph{pinching} if it exists $i\in \{1,\cdots, d\}$ such that $f_i$ has positive Lyapunov exponent on a set of positive Lebesgue measure. The definition of twisting remains the same as definition \ref{def.twisting}, but considering the fixed point given by the constant sequence formed by $i$.

\begin{lemma}\label{l.randomtwist}
Consider the random product formed by $(f_1, \cdots, f_d) \in \mathrm{Diff}_{\leb}^r(S)^d$ and probability $p$. If $(f,Df_c)$ is pinching, then it exists $(g_1, \cdots, g_d) \in \mathrm{Diff}^r_{\leb}(S)^d$ arbitrarily $C^r$-close to $(f_1, \cdots, f_d)$ such that $(g,Dg_c)$ is twisting.
\end{lemma}
\begin{proof}
The only difference from the proof of this lemma and lemma \ref{l.twisting.dense.homeo} is that we want to produce the twisting property just by perturbing one of the diffeomorphisms considered in the random product. Since $(f,Df_c)$ is pinching, it exists $i\in \{1,\cdots , d\}$ such that $f_i$ has positive Lyapunov exponent on a set of positive measure. Consider the fixed point $\tp = (\tp_j)_{j\in \integer}$, where $\tp_j = i$ for every $j\in \integer$. Fix $l\in \{1, \cdots, d\}$ such that $l\neq i$ and consider the point $\tilde{z} = (\tz_j)_{j\in \integer}$, where $\tz_1 = l$ and $\tz_j = i$ for $j\neq 1$. Observe that $\tz$ is a homoclinic point of $\tp$. 

Take $K \subset \{\tp\}\times S$, $t\in K$, $t_{\tz} = h^u_{\tp, \tz}(t)$ and $V\subset \{\tz\} \times S$ be as in the proof of lemma \ref{l.twisting.dense}. Notice that since the holonomies are the identity, $t_{\tz} = t$. Let $V' \subset V$ and $\varphi : D \to V'$ be a $C^r$-diffeomorphisms such that $\varphi(0) = t$. Using polar coordinates, consider $X(r',\theta) = r' \frac{\partial}{\partial \theta}$. Take $\rho: [0,1] \to [0,1]$ a bump function as in the proof of lemma \ref{l.twisting.dense} and consider the vector field $X'(r', \theta) = \rho(r')X(r',\theta)$. For each $t\in \real$, let $\phi_t$ be the time $t$ of the flow on $D$ generated by $X'$. Define a flow on $V'$ by
\[
\Phi_t = \varphi \circ \phi_t \circ \varphi^{-1}.
\]  
Since $\Phi_t$ is the identity on a neighborhood of the boundary of $V'$, we can extend it to a flow on $S$, which is the identity outside $V'$. We also denote this flow by $\Phi_t$. For each $t\in \real$ consider the diffeomorphism $g_l^t = f_l \circ \Phi_t$. For $t$ small, $g_l^t$ is $C^r$-close to $f_l$. By similar arguments as in the proof of lemma \ref{l.twisting.dense}, we can find $t\in \real$ arbitrarily small such that the random product formed by $(g_1 \cdots, g_d)$, with $g_j= f_j$ if $j\neq l$ and $g_l = g^t_l$, is twisting.   
\end{proof}
By theorem \ref{thm.ly17}, the pinching condition is $C^1$-dense in $\mathrm{Diff}^r_{\leb}(S)^d$. Using theorem \ref{t.pin.accum.homeo} and following the same steps in the conclusion of the proof of theorem $A$ we conclude the proof of theorem $C$. 

\section{The proof of theorem D}
\label{sectd}
Notice that in the proof of theorems $A$, $B$ and $C$, the only place where we use $C^1$-perturbations is to use theorem \ref{thm.ly17}, which states that by arbitrarily small $C^1$-perturbation one can get the pinching condition. All the other perturbations that we make are $C^r$. In a recent work, Berger and Turaev proved the following interesting result.
\begin{theorem}[Theorem $A$ in \cite{BeTur}]
\label{thm.bt}
For any surface $S$, if a diffeomorphism $f\in \mathrm{Diff}^{\infty}_{\leb}(S)$ has a periodic point which is not hyperbolic, then there is an arbitrarily small $C^{\infty}$-perturbation of $f$ such that the perturbed map $g\in \mathrm{Diff}^{\infty}_{\leb}(S)$ has a set of positive measure with positive Lyapunov exponents. 
\end{theorem}

In the setting of theorem $D$, if $f_{\tp}$ has an elliptic periodic point, since elliptic periodic points are robust and using that $\mathrm{Diff}^{\infty}_{\leb}(S)$ is $C^r$-dense in $\mathrm{Diff}^r_{\leb}(S)$ (see \cite{Av09-reg}), then $C^r$-close to $f_{\tp}$ there is a diffeomorphism with positive exponents in a set of positive measure. With the same constructions as before we obtain that in this scenario, arbitrarily $C^r$-close to $f$ there is a map which is pinching and twisting, which concludes the proof of theorem $D$.

\section{$u$-states are closed}\label{s.u-closed}
This section is devoted to prove the $u$-states converge to $u$-states. This result maybe of independent interest we prove a more general version here. During the preparation of this paper a proof of this fact was given by Liang, Marin and Yang~\cite{LMY2} for the volume preserving case using recent results about entropy~\cite{AliYang}. Our proof is more direct and does not use entropy.

Let $f_k:M\to M$ be sequence of partially hyperbolic maps admitting linear holonomies converging $C^1$ to $f$ and $E$ be a fiber bundle over $M$ with smooth fibers $\cE$. 
We say that a sequence of cocycles $F^k:E\mapsto E$ over $f_k$ converges to $F:E\mapsto E$ if $F^k$ converges $C^0$ to $F$ and the holonomies converges uniformly in local unstable manifolds.

By remark~\ref{rem.continuity.holon} if $f_k\in \SPV$, $f_k\to f$ in $C^1$ and have the $C^{1+\alpha}$ norm uniformly bounded then $F_k\to F$.

This kind or results was proved for other particular cases, for example~\cite{AliYang} where they proved it when $f_k$ are equal to a fixed Anosov map, \cite{Pol16} where this is proved for $f_k$ a fixed partially hyperbolic map. The principal dificulty here is that the base map is no longer fixed.

We want to prove that if probability measures $m^k$ on $E$ are $F^k$-invariant $u$-states and $m^k\to m$ then $m$ is an $u$-state for $F$.

First let us focus on the smooth case. We actually prove a more general version than we need here, we prove that the $u$-states are closed for every volume preserving partially hyperbolic map (not necessarily a skew product). 

The proof will be an adaptation of the proof of \cite[Theorem~A.1]{Pol16} where the result was proved when the base map is fixed and the cocycle is independent of the base map.

Take $H=[-1,1]^{d_{u}}$ and $V=[-1,1]^{d_{cs}}$ lets write $(x,y)\in [-1,1]^d$ where $x\in H$ and $y\in V$.
For each $p\in M$ and $f$ $C^1$ partially hyperbolic diffeomorphism we call a H\"older continuous map $\Theta_f:[-1,1]^d\to \cV\subset M$ \emph{foliated chart} centered at $p$, if
\begin{itemize}
\item  $\Theta_f(H\times \{y\})\subset W^{uu}_{loc}(\Theta_f(0,y))$,
\item $f\mapsto \Theta_f$ is continuous in the $C^0([-1,1]^d)$ topology in a neighborhood of $f$,
\item for each $y\in V$, $\Theta_{f,y}:=\Theta_f\mid_{H\times \{y\}}$ is $C^1$ and $f\mapsto \Theta_{f,y}$ is continuous in the $C^r(H)$ topology in a neighborhood of $f$,
\item $\Theta_f\mid_{\{0\}\times V}$ is $C^1$, moreover we can take $\Theta_f(\{0\}\times V)=\Gamma$ constant in a neighborhood of $f$.
\end{itemize}

\begin{lemma}\label{l.volpres}
Let $f:M\to M$ be partially hyperbolic $C^{1+\alpha}$, $\alpha>0$, diffeomorphism and $\Theta:H\times V\to \cV$ be a foliated chart, then the pull back of the Lebesgue measure $\eta=\Theta^{-1}_*\vol\mid_\cV$ has the form $\rho_f dxdy$, where $dxdy$ is the usual volume form in $H\times V$ and $\rho_f$ varies continuously with $f$ with respect to the $C^1$ norm in subsets of $C^{1+\alpha}$ bounded norm.
\end{lemma}
\begin{proof}
Let $\nu=\vol\mid_\cV$ and take the disintegration of $\nu$ into local unstable manifolds $\nu=\nu_{y'}d\nu_\Gamma$ where $y'\in \Gamma=\Theta(\{0\}\times V)$, by the absolute continuity of the unstable foliation $\nu_\Gamma$ is the (normalized) lebesgue measure on $\Gamma$ and $\nu_{y'}=\varrho_f vol_{W^u(y')}$ where $\varrho_f$ is defined up to normalization by 
\begin{equation}\label{eq.densitiy}
\frac{\varrho_f(x')}{\varrho_f(y')}=\lim_{n\to -\infty}\frac{\det Df^n\mid_{ W^u(x')}}{\det Df^n\mid_{W^u(y')}}
\end{equation}
moreover the limit in \eqref{eq.densitiy} is uniform in subsets of bounded $C^{1+\alpha}$ norm. So, normalizing as $\nu_{y'}(H\times y')=1$, $\varrho_f$ depends continuously on $f$ with respect to the $C^1$ topology in subsets of bounded $C^{1+\alpha}$ norm.

Disintegrating $\eta$ in horizontals we have $\eta=\eta_y d \eta_V$, where $\eta_y=\Theta^{-1}_*\nu_{\Theta(y)}$ and $\eta_V=\Theta^{-1}_*\nu_\Sigma$.

Let $\varphi:H\to \real$ be a continuous function, then 
$$
\begin{aligned}
\int \varphi d\eta_y&=\int \varphi\circ \Theta^{-1}(t)\nu_{\Theta(y)}(t)\\
&=\int \varphi\circ \Theta^{-1}(t)\varrho_f(t)d \vol_{W^u(\Theta(y))}(t)\\
&=\int \varphi \abs{\det \Theta_y(t)}\varrho_f(\Theta(t))d x
\end{aligned}.
$$

So $\eta_y=\rho'_f(\cdot,y)dx$ where $\rho'_f$ is continuous and depends continuously on $f$. Analogously $\eta_V=\rho''_f dy$,
then by Fubini $\eta=\rho_f dxdy$, with $\rho_f$ continuous on $f$.
\end{proof}

Now lets define our foliated charts for the skew products over hyperbolic homeomorphisms.

Fix a local chart given by the product structure 
$$
\theta:W^s_{loc}(\tx_0)\times W^u_{loc}(\tx_0)\to \cV\subset \tM\text{, }\theta(x,y)=[x,y].
$$

Let $x^u\in W^u_{loc}(\tx_0)$ be fixed, for any $\tx=(x^-,x^+)\in W^s_{loc}(\tx_0)\times W^u_{loc}(\tx_0)$, let $\phi(\tx) = [x^-,x^u]$ and then define
\begin{equation*}
\Theta(\tx,t)=\big(\tx,h^{u}_{\varphi(\tx),\tx}(t)\big).
\end{equation*}

This local chart sends $\{x^-\}\times W^u_{loc}(\tx_0)\times \{t\}$ to $W^{uu}_{loc}(x^-,x^u,t)$. Take 
$H=W^u_{loc}(\tx_0)$ and $V=W^s_{loc}(\tx_0)\times S$, then the foliated chart is given by $\Theta:H\times V\mapsto \cV$.

As $h^u_{\varphi(\tx),\tx}$ preserves the volume on $S$ the next lemma follows
\begin{lemma}\label{l.hyphom}
Let $\sigma:\Sigma\to\Sigma$ be a hyperbolic homeomorphism with invariant measure $\hmu$ and $f\in \SP(\Sigma\times S)$, then 
$\Theta^* \tmu\times \leb=\tmu\times \leb$.
\end{lemma}

Now we can prove that $u$-states converges to $u$-states in both cases. Let $f_k$ be a partially hypebolic volume preserving map or $f_k\in \SP(\Sigma\times S)$, that converges to $f$ in the corresponding topology.
We have our theorem

\begin{theorem}\label{t.closedu}
Let $F_k\to F$ and $m^k$ be $u$-states converging weakly to $m$, then $m$ is an $u$-state for $F$.
\end{theorem}
\begin{proof}
Let $\mu$ be the lebesgue measure if we are in the partially hyperbolic volume preserving case or $\hmu\times \leb$ in the other case. 

Denote by $\Theta_k:H\times V\to \cV$ the foliated chart for $f_k$ and let $\eta^k=\Theta_k^*\mu\mid_\cV$. By lemmas~\ref{l.volpres} and \ref{l.hyphom} we can write 
$\eta^k=\rho_k \eta$ for some $\rho_k:H\times V\to \real$, where $\eta=\Theta^*\mu\mid_\cV$, moreover there exist $0<c<C$ such that $c\leq \rho_k \leq C$.
 
Observe that it is sufficient to prove the result locally, so we can suppose that the projective tangent bundle is locally a product $\cV\times \cE$.

Define $\cH_k:H\times V\times \cE\to H\times V\times \cE$ as
$$
(x,y,v)\to (x,y,H^{u,k}_{\Theta_k(x,y),\Theta_k(0,y)}(v))
$$

Let $\hat{m}^k=\left(\Theta_k^{-1}\times \id\right)*m^k\mid_{\cV\times \cE}$, define $\tm^k={\cH_k}_*\hm^k$. Observe that $m^k$ is a $u$-state if and only if the disintegration of $\tm^k$ with respect to the projective fibers is such that 
\begin{equation}\label{eq.ustate}
\tm^k_{(x,y)}=\tm^k_{(x',y)}\text{ for }\eta\text{-almost every }(x,y),(x',y)\in H\times V.
\end{equation}

Denote by $\nB_0$ the sigma algebra of sets $H\times \{y\}$, $y\in V$, then \eqref{eq.ustate} is equivalent to $p\mapsto \tm^k_p$ being $\nB_0$ measurable.

So we are left to prove that $m^k\to m$ implies that $p\mapsto \tm_p$ is also $\nB_0$ measurable.

We follow the proof of \cite[Proposition~A.7]{Pol16} where this was proved for cocycles with a fixed base map. In order to do this we first need to adapt Lemma~A.3 and Lemma~A.4 of \cite{Pol16} to our scenario.

\begin{lemma}[Addaptation of Lemma~A.3]\label{l.measurable1}
Let $\phi:H\times V\times\cE\to \real$ be a measurable bounded function such that for every $p\in H\times V$, $v\mapsto \phi(p,v)$ is continuous, then $\int \phi d \hm^k\to \int \phi d \hm$.
\end{lemma}
\begin{proof}
To simplify the notation let $\hTheta_k:=\Theta_k\times \id$.
Recall that ${\Theta_k^{-1}}_*\mu_k\mid_\cV=\rho_k\eta$ is the projection of $\hm^k=(\hTheta_k)_* m^k$.

We have
$$
\begin{aligned}
\abs{\int \phi\circ \hTheta_k^{-1}  d m^k- \int \phi\circ \hTheta^{-1} d m}\leq\\
 \abs{\int \phi\circ\hTheta_k^{-1} - \phi\circ \hTheta^{-1} d m^k}+\abs{\int \phi\circ \hTheta^{-1}  d m^k- \int \phi\circ \hTheta^{-1} d m}.
\end{aligned}
$$

By \cite[lemma~A.3]{Pol16} the second term goes to zero when $k\to\infty$, so we focus on the first term.

Fix $\varepsilon>0$ and take a compact set $K\subset H\times V$ such that $\eta(H\times V\setminus K)<\frac{\varepsilon}{C\norm{\phi}}$ and $\phi$ is continuous in $K\times \cE$. Take   $\phi':H\times V\times \cE \to \real$ a continuous function such that $\phi(p,v)=\phi'(p,v)$ for every $p\in K$, $v\in \cE$ and $\norm{\phi'}\leq \norm{\phi}$, then 

$$\begin{aligned}\abs{\int \phi\circ \hTheta_k^{-1} - \phi\circ \hTheta^{-1} d m^k}\leq\\
 \int\abs{ \phi'\circ \hTheta_k^{-1} - \phi'\circ \hTheta^{-1}} d m^k+\int\abs{ \phi' - \phi} d (\hTheta_k^{-1})_* m^k +\int\abs{ \phi' - \phi} d \hTheta^{-1}_* m^k.
\end{aligned}$$

The first term goes to zero by continuity of $\phi'$ and the second and last term are bounded by $\sup\{\rho_k\}\eta(H\times V\setminus K)<\epsilon$.
\end{proof}
\begin{lemma}[Addaptation of Lemma~A.4]
If $m^k\to m$ then $\tm^k\to \tm$.
\end{lemma}
\begin{proof}
Take $\phi:H\times V\times \cE\to \real$ uniformly continuous, then 
$$
\begin{aligned}
\abs{\int \phi d \tm^k-\int \phi d\tm}\leq \\
 \abs{\int \phi\circ \cH d \hm^k-\int \phi\circ \cH d \hm}+\int \abs{\phi\circ \cH_k - \phi\circ \cH }d \hm^k
\end{aligned}
$$
The first term goes to zero by lemma~\ref{l.measurable1} and the second one by uniform convergence of the holonomies.
\end{proof}
Now the \cite[Proposition~A.7]{Pol16} can be directly addapted. So we have that $p\mapsto \tm_p$ is $\nB_0$ measurable concluding the proof. 
\end{proof}
\begin{remark}
Observe that Theorem~\ref{t.closedu} can be addapted to a non conservative setting, when $\mu_k$ are $f_k$-invariant probabilities converging to an $f$-invariant probability $\mu$ such that $\eta^k=\Theta_k^*\mu_k\mid_\cV$ has the form
$\eta^k=\rho_k \eta$ for some $\rho_k:H\times V\to \real$ and there exist $0<c<C$ such that $c\leq \rho_k \leq C$.
\end{remark}

\section{Further applications and questions}
In this section we explain some further applications of our techniques and some conjectures.

\subsection{$C^r$ density}\label{ss.cr-density}
As we mentioned in section \ref{sectd}, the only place where we use $C^1$-perturbations is for theorem \ref{thm.ly17}, which in the setting of theorem $D$ is replaced by theorem \ref{thm.bt}. Both theorems are used to obtain the pinching property. So to prove $C^r$ density of positive exponents using our techniques we have to answer the following question:
\emph{Is the pinching property $C^r$ dense? Or is $C^r$ dense in some sub-class of diffeomorphisms?}
\subsection{Higher dimension}
In the whole work we assume that $S$ is a $2$-dimensional manifold. First because in the two dimensional case positive exponents in a positive measures set implies non uniformly hyperbolicity in this set, in higher dimension using our techniques we can only hope to have at least one positive exponents.

As a technicall reason, we use the assumption in the proof of Lemma~\ref{l.twisting.dense} to do the perturbation, we belive that this can be addapted to a higher dimensional case (see~\cite{Pol16} where a higher dimensional case is treated for linear cocycles). Also we use the $2$-dimensional assumption in Lemma~\ref{lem.smoothpath}, we do not know if this lemma is true un for higher dimensional manifolds.

In the case o random product, as the perturbations are locally constant we do not need the homotopy constructed in Lemma~\ref{lem.smoothpath}, so we conjecture that 
\begin{conjecture}
Let $S$ be a manifold of dimension greater than $2$, then there exist a $C^1$ dense and $C^r$ open set of volume preserving diffeomorphisms such that the random product has at least one positive exponent.
\end{conjecture} 

\information

\end{document}